\documentclass{amsart}

\usepackage{algorithm}
\usepackage{algorithmic}
\usepackage{amssymb,amsmath,amsthm}

\theoremstyle{plain}
  \newtheorem{theorem}{Theorem}[section]
  \newtheorem*{theorem*}{Theorem}
  
  \newtheorem*{corollary*}{Corollary}
  \newtheorem{lemma}[theorem]{Lemma}
  \newtheorem*{lemma*}{Lemma}
  \newtheorem{proposition}[theorem]{Proposition}
  \newtheorem*{proposition*}{Proposition}

\theoremstyle{definition}

  \newtheorem*{assumption*}{Assumption}
  
  \newtheorem*{conjecture*}{Conjecture}
  \newtheorem{definition}[theorem]{Definition}
  \newtheorem*{definition*}{Definition}
  
  \newtheorem*{example*}{Example}
  
  \newtheorem*{hypothesis*}{Hypothesis}
  
  \newtheorem*{property*}{Property}
  
  \newtheorem*{remark*}{Remark}

\newcommand{\E}{{\rm E} }

\newcommand{\Z}{ \mathbb{Z}}  
\newcommand{\N}{\mathbb{N}}
\newcommand{\Ot}{ \widetilde{O}}  

\newcommand{\lam}{\lambda}

\renewcommand{\vec}[1]{\mbox{\boldmath$#1$}}
\newcommand{\floor}[1]{\lfloor #1 \rfloor}

\begin{document}

\title{Division algorithms for the fixed weight subset sum problem}

\author[Andrew Shallue]{Andrew Shallue}
\address{Department of Mathematics and Computer Science, Illinois Wesleyan University, Bloomington, Illinois}
\email{ashallue@iwu.edu}
\thanks{Research supported by ICORE}



\begin{abstract}
Given positive integers $a_1, \dots, a_n, t$, the fixed weight
subset sum problem is to find a subset of the $a_i$ that sum to $t$, 
where the subset has a prescribed number of elements.  It is this problem
that underlies the security of modern knapsack cryptosystems, and 
solving the problem results directly in a message attack.  We present
new exponential algorithms that do not rely on lattices, and hence
will be applicable when lattice basis reduction algorithms fail.
These algorithms rely on a generalization of the notion of splitting system
given by Stinson \cite{Stin02}.
In particular, if the problem has length $n$ and weight $\ell$ then 
for constant $k$ a power of two less than $n$
we apply a $k$-set birthday algorithm to the splitting system of the problem.
This randomized algorithm has time and space complexity that satisfies 
$T \cdot S^{\log{k}} = \Ot({n \choose \ell})$ (where the constant depends 
uniformly on $k$).  In addition to using space 
efficiently, the algorithm is highly parallelizable. 
\end{abstract}

\maketitle

\emph{Author's Foreword - January 2012}
\bigskip

While the present paper was being refereed,  \cite{HGJoux10} came out with an 
improvement to the main result.  The most interesting aspect that remains is the 
idea of a $k$-set splitting system.

\section{Problem Statement}

Let $a_1, \dots, a_n$ and a target $t$ be positive integers.  The $\ell$-weight subset sum
problem is to find a subset of the $a_i$ that sum to $t$, where the subset has $\ell$ elements.
Equivalently, the problem is to find a bit vector $\vec{x}$ of length $n$ and Hamming weight $\ell$
such that 
\begin{equation}
\sum_{i=1}^n a_i x_i = t \enspace .    \label{statement}
\end{equation}
The corresponding decision problem is to determine whether or not a solution exists.
We will refer to  the {\it integer} subset sum problem as seeking a solution for 
(\ref{statement})
over the integers, while solving the {\it modular} subset sum problem involves solving
(\ref{statement}) over some ring $\Z/m\Z$.  A modular subset sum problem is {\it random} if 
we assume that the $a_i$ are chosen uniformly at random from $\Z/m\Z$.

The most important quantity associated with a subset sum problem is its \emph{density}, 
defined to be $\frac{n}{\log{A}}$ in the integer case where 
$A = {\rm max}_{1 \leq i \leq n} a_i$.  In the modular case we define density to be 
$\frac{n}{\log{m}}$ and will refer to it as \emph{modular density}.
Inspired by \cite{Kun08}, we define the \emph{information density} 
to be $\frac{\log{{n \choose \ell}}}{\log{A}}$ (for the integer case) and 
\emph{modular information density} to be $\frac{\log{{n \choose \ell}}}{\log{m}}$.  

The fixed weight subset sum problem is interesting both because it is NP-complete
and because it has applications to knapsack cryptosystems (see Section \ref{sec:knapsack}).
A brute force attack on the fixed weight subset sum problem takes $\Ot({n \choose \ell})$
bit operations.  Here $\Ot$ is ``Soft-Oh'' notation.  For functions $f$ and $g$, 
we say $f$ is $\Ot(g)$ if there exist $c, N \in \N$ such that 
$f(x) \leq g(n)(\log(3+g(n)))^c$ for all $n \geq N$.

Throughout this paper all logarithms will have base $2$.
Suppose $L$ is a set of integers and $a$ is an integer.  Then $L-a$ is the 
set given by $\{b - a : b \in L\}$ and $L-a \mod{m}$ is the set given 
by $\{b - a \mod{m}: b \in L\}$.

\section{Prior Work and New Results}\label{sec:results}



It is a nontrivial matter to apply the standard algorithmic technique of divide-and-conquer
to problems with fixed weight bit vectors.  One solution is to employ a $k$-set splitting
system.  Throughout most of this paper we assume that $n$ and $\ell$ are divisible by $k$.
See Section \ref{sec:gen_splitting} for a discussion of the general case.

\begin{definition}
An \emph{$(n, \ell, k)$-splitting system} is a set $X$ of $n$ indices along with a set $\mathcal{D}$
of divisions, where each division is itself a set $\{I_1, \dots, I_k\}$ of subsets of indices, 
with $I_1 \cup \dots \cup I_k = X$ and $| I_1 | = \dots = | I_k | = n/k$.  These objects have 
the property that for every $Y \subseteq X$ such that $|Y| = \ell$, there exists a division
$\{I_1, \dots, I_k\} \in \mathcal{D}$ such that $|Y \cap I_j| = \ell/k$ for $1 \leq j \leq k$.
We call this division a good division with respect to $Y$.
\end{definition}

All splitting systems will appear in the context of a 
fixed weight subset problem with unknown solution $Y$.  With $n$ and $\ell$ 
understood from context, we will refer to an $(n, \ell, k)$-splitting system as a 
$k$-set splitting system.  With $Y$ understood from context, we will call 
a division such that $|Y \cap I_j| = \ell/k$ for all $I_j$ a good division.

This is a generalization of $2$-set splitting systems presented by Stinson in \cite{Stin02},
which he called $(N; n, \ell)$-splitting systems.  In that paper design theory was 
utilized to minimize $N$, the number of divisions. 

Two set splitting systems allow for the application of the baby-step-giant-step 
algorithm to attain a square root time-space
tradeoff.  This had been done before Stinson, but without formalizing 
the notion of splitting systems.  A version of this algorithm that searches for 
a good division randomly is presented 
in \cite[Section 7.3]{ChRiv88} and applied to the fixed weight subset sum problem
as a message attack against knapsack cryptosystems.  Coppersmith developed the 
same algorithm for use on the fixed weight discrete logarithm problem, as well as 
a version that found a good $2$-division deterministically rather than randomly.
Both are presented in \cite{Stin02} along with an average case analysis.

Another line of attack on the fixed weight subset sum problem 
was revealed by the work of Nguyen and Stern \cite{NS05}.  
They modified the lattice basis reduction technique of \cite{Cos92} to also
work for problems of small pseudo-density $\frac{\ell \log{n}}{\log{A}}$.  
Thus problems can be reduced to the closest vector problem on lattices.
In practice this means that any problem with information density less than one
and $n$ less than $300$ or so can be solved by current lattice 
reduction algorithms.
 
We present new algorithms for the fixed weight subset sum problem, which in 
the case of Theorem \ref{thm:SS} is also a new algorithm for the fixed weight discrete 
logarithm problem.  
We use $T$ and $S$ to refer to the exponential term of an algorithm's time 
and space usage.


\begin{theorem}\label{thm:SS}
There is an algorithm for the fixed weight subset sum problem whose time and 
space constraints lie on the curve $T \cdot S^2 = {n \choose \ell}$.
The deterministic version takes $\Ot(n^3 {n/4 \choose \ell/4}^2)$ bit operations
and the randomized version is expected to take 
$\Ot(\ell^{3/2} {n/4 \choose \ell/4}^2)$ bit operations.  Both have space 
complexity $\Ot({n/4 \choose \ell/4})$.
\end{theorem}

\begin{theorem}\label{thm:kset}
Choose parameters $m$ and $k$ so that $k$ is a power of $2$, 
$m < {n/k \choose \ell/k}$, and $\log{m} \geq 2(\log{k})^2$.
Assume that when reduced modulo $m$, 
the $a_i$ are uniformly random elements of $\Z/m\Z$.
Then there is a randomized algorithm for the fixed weight subset sum problem 
whose expected running time is
$\Ot(m^{1/(\log{k}+1)} \cdot {n \choose \ell}/m)$ and which uses 
$\Ot(m^{1/(\log{k}+1)})$  space.  This gives a point on the time/space tradeoff curve
$T \cdot S^{\log{k}} = {n \choose \ell}$.
\end{theorem}

Note that the assumption $m < {n/k \choose \ell/k}$ implies that the modular information 
density is greater than $k$.  Also note that random fixed weight subset sum problems require 
information density greater than one to ensure a solution exists with high probability.
This makes Theorem \ref{thm:kset} a counterpoint
to the lattice reduction technique employed in \cite{NS05}.  Finally note that 
the hidden polynomial terms include $\Theta(\ell^{\frac{1+k}{2}})$, 
 the expected cost of finding a good $k$-division 
 (See Proposition \ref{rand_splitting}).  This limits the applicability of 
Theorem \ref{thm:kset} to practical settings.

The key ingredient of the first theorem is the general decomposition algorithm 
of Schroeppel-Shamir \cite{SchSh81}, while the second theorem relies on the 
$k$-set birthday algorithm of Wagner \cite{Wagner02}.  The application of 
these algorithms to the fixed weight setting relies on splitting systems to perform
the necessary decomposition.
Another candidate for the $k$-division algorithm is the generalization of 
Schroeppel-Shamir outlined in \cite{Vys87} (but see \cite{Fer93} for a rebuttal).

The general idea behind the algorithm of Theorem \ref{thm:kset} is the following.
We pick a parameter $m$ so that the corresponding modular problem has 
high enough modular density for the $k$-set birthday algorithm to be successful.  Noting
that the sought for integer solution is included in the set of solutions to the modular 
problem, we construct a modular oracle which outputs one of the modular solutions
(nearly) uniformly at random.  By repeating the modular oracle enough times, we expect
to eventually find a solution to the original problem over the integers.  The choice
of $m$ determines the point on the time-space tradeoff curve, with larger choices 
being better in the sense that $T$ is smaller.

The importance of this new work is in improving the space complexity of the fixed 
weight subset sum problem.  Theorem \ref{thm:SS} is a direct improvement of the 
work given in \cite{Stin02}, while Theorem \ref{thm:kset} is the first to give a 
time/space tradeoff curve better than $T \cdot S^2$ for this problem.  Although the 
time bound for the algorithm of Theorem \ref{thm:kset} will nearly always be 
worse than $\Ot({n/2 \choose \ell/2})$ due to the limitations on the choice of $m$, the 
algorithm is highly parallelizable by simply running the modular oracle on 
several processors at once.  Thus with enough processors each will have less 
than $\Ot({n/2 \choose \ell/2})$ work to do.  An interesting open problem is to 
generalize the work in \cite{MinAlis09} to the subset sum problem, and then to 
explore the improvements to Theorem \ref{thm:kset} that result from loosening 
the upper bound on $m$.

In Section \ref{sec:splitting} we present the notion of an $(n,\ell,k)$-splitting system
and prove that they exist assuming $k$ divides $n$ and $\ell$.  
We prove Theorem \ref{thm:SS} in Section \ref{sec:SS},
develop the modular oracle in Section \ref{sec:mod_oracle}, and prove 
Theorem \ref{thm:kset} in Section \ref{sec:kset}.
We discuss application to message attacks on knapsack cryptosystems 
in Section \ref{sec:knapsack}, 
experimentally seek the optimal choice of $m$ in Section \ref{sec:data}, and 
finish by proving $(n, \ell, k)$-splitting systems exist in general in Section \ref{sec:gen_splitting}.

\section{Splitting Systems}\label{sec:splitting}

Recall the definition of $k$-set splitting system given in the previous section, and that 
for now we assume both $n$ and $\ell$ are divisible by $k$.
In \cite{Stin02} it is proved that the probability of a random $2$-division being good 
is $\Omega(\ell^{-1/2})$ and that there is a trivial construction that yields a $2$-set 
splitting system with $n$ divisions.  In this section we generalize these results 
for $k$-set splitting systems.  Note that design theory may yield a 
construction of a $k$-set splitting system with fewer divisions, as Stinson showed 
that a $2$-set splitting system exists with at most $\ell^{3/2}$ divisions in \cite{Stin02}.

The first result 
is a polynomial bound on the probability of choosing a good $k$-division randomly.  
One important note is that the 
constant depends exponentially on $k$, so it is important that $k$ be a fixed parameter.

\begin{proposition} \label{rand_splitting}
The probability of choosing a good $k$ division is bounded below by a constant times
$\ell^{\frac{1-k}{2}}$.
\end{proposition}
\begin{proof}
First consider the number of ways of choosing $k$ sets of $n/k$ items from 
a total of $n$ items.  It is
$$
\frac{1}{k!} {n \choose n/k} {n-n/k \choose n/k} \cdots {n-(k-2)n/k \choose n/k}
= \frac{1}{k!} \frac{n!}{(\frac{n}{k}!)^k} 
$$
where the extra $\frac{1}{k!}$ term offsets the double counting that results
from the $k$ sets being indistinguishable.

This is also the number of $k$ divisions.  The number of good $k$ divisions
is counted by choosing $k$ equal sized sets from $Y$ and choosing $k$ equal 
sized sets from $X \setminus Y$.  Thus
the probability of choosing a good $k$-division is 
\begin{equation}
\frac{1}{k!} \cdot 
\frac{\ell !}{ ( \frac{\ell}{k} !)^k } \frac{(n-\ell)!}{(\frac{n-\ell}{k}! )^k} 
\left/ \frac{n!}{( \frac{n}{k}! )^k} \right.
\enspace .     \label{k_prob}
\end{equation}

We next find upper and lower bounds on $(n!)/(\frac{n}{k}!)^k$. 
Stirling's formula gives us
$$
2n^ne^{-n} \sqrt{2 \pi n} \geq n! \geq n^n e^{-n} \sqrt{2 \pi n} \enspace .
$$
For the lower bound this implies
$$
\frac{n!}{(\frac{n}{k}!)^k} \leq 
\frac{2 n^n e^{-n} \sqrt{2 \pi n}}{( (\frac{n}{k})^{n/k} e^{-n/k} \sqrt{\frac{2}{k} \pi n})^k} =
k^n \cdot 2 k^{k/2} (2 \pi n)^{\frac{1-k}{2} }
$$
while similarly for the upper bound we have
$$
\frac{n!}{(\frac{n}{k}!)^k} \geq 
\frac{n^n e^{-n} \sqrt{2 \pi n}}{( 2(\frac{n}{k})^{n/k} e^{-n/k} \sqrt{\frac{2}{k} \pi n})^k} =
k^n \cdot 2^{-k} k^{k/2} (2 \pi n)^{\frac{1-k}{2} } \enspace .
$$

Thus (\ref{k_prob}) has a lower bound given by
$$
\frac{1}{k!} \cdot \frac{k^{\ell} \cdot 2^{-k}k^{k/2} (2 \pi \ell)^{\frac{1-k}{2}} \cdot
	 k^{n-\ell} \cdot 2^{-k} k^{k/2} (2 \pi (n-\ell))^{\frac{1-k}{2}} }
{k^n \cdot 2 k^{k/2} (2 \pi n)^{\frac{1-k}{2}} }
\geq c \cdot \ell^{\frac{1-k}{2}} 
$$
for some constant $c$ that does not depend on $\ell$ or $n$, 
but does depend exponentially on $k$.
\end{proof}

Next we construct a $k$-set splitting system with fewer than $n^{k-1}$ divisions, 
showing that a good division can be found deterministically in fewer 
than $n^{k-1}$ trials.  This requires 
first proving that one of the sets 
$B_i = \{i+j \mod{n} \ | \ 0 \leq j \leq n/k \}$
satisfies $|B_i \cap Y| = \ell/k$.

\begin{proposition}\label{find_one}
Let $Y$ be a subset of $\{0, \dots, n-1\}$ of size $\ell$.  Then there exists $B_i$ such that
$|B_i \cap Y| = \ell/k$.
\end{proposition}
\begin{proof}
First note that $B_0, B_{n/k}, B_{2n/k}, \dots, B_{(k-1)n/k}$ partition the set of $n$ indices.
Now if $|B_0 \cap Y| = \ell/k$ we are done, so instead suppose 
(without loss of generality) that $|B_0 \cap Y| > \ell/k$.  
Then since we have a partition above, one of
the $B_i$ for $i = 0, \frac{n}{k}, \frac{2n}{k}, \dots, \frac{(k-1)n}{k}$ 
must have the property that 
$|B_i \cap Y| < \ell/k$.  Call this set $B_j$.

Define a function $v$ by 
$v(i) = |B_i \cap Y| -\ell/k$ and note that $| v(i) - v(i+1) | \leq 1$.  Since $v(0) > 0$ and 
$v(j) < 0$, there must be some $i$ with $v(i) = 0$.  This completes the proof.
\end{proof}

The construction now follows by finding each $I$ in turn.

\begin{proposition}\label{det_splitting}
There exists a $k$-set splitting system with fewer than $n^{k-1}$ divisions.
\end{proposition}
\begin{proof}
By Proposition \ref{find_one}, there exists a $B_i$ such that $| B_i \cap Y | = \ell/k$.
Call it $I_1$, and reorder the $a_i$ so that the indices in $I_1$ are the last $n/k$ indices.

Redefine the $B_i$ so that they still have size $n/k$, but now wrap modulo $n - n/k$ rather 
than $n$.  Proposition \ref{find_one} is still valid, and so there exists a $B_i \subseteq
X \setminus I_1$ such that $| B_i \cap Y | = \ell/k$.  Call it $I_2$, and reorder the $a_i$ so 
that the indices in $I_2$ are the last $n - n/k$ indices.

By continuing in this fashion, we find a good division.  Only $I_1, \dots, I_{k-1}$
need to be searched for, since $I_k$ consists of the leftover indices.

The number of divisions is the product of the number of $B_i$ searched
for each of $I_1, \dots, I_{k-1}$, which is
$$
n \left(n- \frac{n}{k} \right)\left(n- \frac{2n}{k}\right) 
\cdots \left(n - \frac{(k-2)n}{k}\right) < n^{k-1} \enspace .
$$
\end{proof}

\section{Applying Schroeppel-Shamir}\label{sec:SS}

Chor and Rivest in  \cite[Section 7.3]{ChRiv88} proposed that the general algorithm 
of Schroeppel and Shamir \cite{SchSh81} may be applicable to the fixed weight subset 
sum problem.  In this section we accomplish this, giving a square root 
time and fourth root space algorithm.  The only missing ingredient 
was the idea of a $4$-set splitting system.  
We will assume for ease of 
exposition  that $n$ and $\ell$ are divisible by $4$.  
See Section \ref{sec:gen_splitting}
for the general case.

We review the theory of problem decomposition presented in \cite{SchSh81}, though we 
specialize to the case of using a good $4$-division to solve the $\ell$-weight subset sum 
problem.  

The fixed weight subset sum problem has length $n$ and weight $\ell$.  By 
Section \ref{sec:splitting} the problem can be decomposed into subproblems
of length $\frac{n}{4}$ and weight $\frac{\ell}{4}$.  As with all 
subset sum problems, this decomposition is sound, complete, and polynomial
(see \cite{SchSh81} for definitions).  However, it is not additive, and thus 
does not satisfy Schroeppel-Shamir's definition of a composition operator.
Fortunately, this lack does not affect the analysis of their algorithm, only the 
expression of the complexity.

In order to apply the Schroeppel-Shamir algorithm, our decomposition must 
have two essential properties.

\begin{definition}
A set of problems is \emph{polynomially enumerable} if there is a polynomial time algorithm
which finds for each bit string $x$ the subset of problems which are solved by $x$.
\end{definition}

\begin{definition}
A composition operator $\oplus$ is \emph{monotonic} if the problems of each size
can be totally ordered in such a way that $\oplus$ behaves monotonically:
if $|P'|=|P''|$ and $P' < P''$ then $P \oplus P' < P \oplus P''$ and 
$P' \oplus P < P'' \oplus P$.
\end{definition}

 Define a problem on set $j$, $1 \leq j \leq 4$ by
$(b, \{a_i \ | \ i \in I_j\})$ where $\vec{x}$ of weight $\ell/4$ is a solution if 
$$
\sum_{i \in I_j} a_i x_i = b \enspace .
$$
Define a composition operator by
$$ P_{j} \oplus P_{j'} = (b+b', \{a_i \ | \ i \in I_j \cup I_{j'} \}) \enspace .$$
This is polynomial and polynomially enumerable since addition is polynomial
time.  It is sound since if 
$\sum_{i \in I_j} a_ix_i = b$ and $\sum_{i \in I_{j'}} a_ix_i = b'$ then 
$\sum_{i \in I_j \cap I_{j'}} a_i x_i = b+b'$.  It is complete by the definition 
of a good division.

Finally, $\oplus$ is monotonic if we order problems by their solution $b$, 
and if this is equal then lexicographically by their
sets $\{a_i \ | \ i \in I_j\}$.
 For suppose that 
$(b', \{a_1', \dots, a_{n/4}'\}) < (b'', \{a_1'', \dots, a_{n/4}''\})$.  Then 
\begin{align*}
&(b'+b, \{a_1', \dots, a_{n/4}', a_1, \dots, a_{n/4}\}) 
< (b''+b, \{a_1'', \dots, a_{n/4}'', a_1, \dots, a_{n/4}\}) \ \mbox{ and} \\
&(b+b', \{a_1, \dots, a_{n/4}, a_1', \dots, a_{n/4}'\}) 
< (b+b'', \{a_1, \dots, a_{n/4}, a_1'', \dots, a_{n/4}''\}) \enspace .
\end{align*}

We now state the main theorem in the context of the $\ell$-weight subset sum problem.

\begin{theorem}[Schroeppel and Shamir \cite{SchSh81}]
If a set of problems is polynomially enumerable and has a monotonic composition
operator, then instances can be solved in time $\Ot({n/4 \choose \ell/4}^2)$ and 
space $\Ot({n/4 \choose \ell/4})$.
\end{theorem}

The algorithm is summarized as follows.  
Let $P$ be a problem of length $n$ and weight $\ell$ for which we seek a solution, 
and assume we are given a good division.  For $I_1, I_2, I_3, I_4$ enumerate
all subproblems and store in tables $T_j$, $1 \leq j \leq 4$.

Sort $T_2$ in increasing order and sort 
$T_4$ in decreasing order.  Make two queues (with arbitrary polynomial time insertions
and deletions), with the first containing pairs 
$(P_1, \mbox{smallest } P_2)$ for all $P_1 \in T_1$ and the other containing
pairs $(P_3, \mbox{largest } P_4)$ for all $P_4 \in T_4$.  Now repeat the following
until either a solution is found or both queues are empty 
(in which case there is no solution):
compute $S = (P_1 \oplus P_2)\oplus (P_3 \oplus P_4)$ and output $S$ 
if $S = P$.
If $S < P$ delete $(P_1, P_2)$ from the first queue and 
add $(P_1, P_2')$ where $P_2'$ is the successor 
of $P_2$.  If $S > P$ delete $(P_3, P_4)$ from the second queue 
and add $(P_3, P_4')$ where $P_4'$
is the successor of $P_4$.

We conclude that if we have a good $4$-division, the algorithm of Schroeppel and Shamir 
will solve the problem.
By Propositions \ref{rand_splitting} and \ref{det_splitting} we know that a good 
$4$-division can be found in $O(n^3)$ trials deterministically or expected 
$O(\ell^{3/2})$ trials randomly.  This inspires the following algorithm for the 
fixed weight subset sum problem.

\begin{algorithm} 
\caption{Schroeppel-Shamir for fixed weight subset sum}
\label{alg:SS}
\begin{algorithmic}[1]
\STATE {\bf Input: } positive integers $a_1, \dots, a_n$, $t$, $\ell$
\STATE {\bf Output: } $\vec{x} \in \{0,1\}^n$ of weight $\ell$ 
such that $\sum_{i=1}^n a_i x_i = t$
\WHILE{ no solution}
\STATE choose division $D = \{I_1, I_2, I_3, I_4\}$ 
\STATE for $1 \leq j \leq 4$ form table $T_j$ of problems, one for 
each weight $\ell/4$ subset of $I_j$
\STATE apply Schroeppel-Shamir to $T_1, T_2, T_3, T_4$
\ENDWHILE
\end{algorithmic}
\end{algorithm}  

\begin{proof}[Proof of Theorem \ref{thm:SS}]
The correctness follows from the monotonicity of $\oplus$, see \cite{SchSh81}
for details. From \cite{SchSh81}, 
the maximum number of elements in either queue at any one time is 
${n/4 \choose \ell/4}$ and the maximum number of steps needed is the 
number of pairs $(P_i, P_j) = {n/4 \choose \ell/4}^2$.  Thus the space complexity 
of Algorithm \ref{alg:SS} is $\Ot({n/4 \choose \ell/4})$ and the time complexity is 
$\Ot(n^3 {n/4 \choose \ell/4}^2)$ using deterministic splitting 
and $\Ot(\ell^{3/2} {n/4 \choose \ell/4}^2)$ using randomized splitting.
\end{proof}

As this work was inspired by Stinson's paper \cite{Stin02}
 on the fixed weight discrete logarithm problem, 
it is worth noting that Algorithm \ref{alg:SS}
applies directly to that problem as well. 

Also note that given a brute force running time of $O({n \choose \ell})$,
Algorithm \ref{alg:SS} is a square root time and fourth root space algorithm, 
and hence lies on the tradeoff curve $T \cdot S^2 = {n \choose \ell}$.
This is justified by Stirling's formula, which gives 
$$
{n/4 \choose \ell/4} = 
\Theta \left({n \choose \ell}^{1/4} \left( \frac{n}{\ell(n-\ell)} \right)^{3/8} \right) \enspace .
$$

\section{Modular Oracle}\label{sec:mod_oracle}

Having proved Theorem \ref{thm:SS}, our task in the next two sections is 
to prove Theorem \ref{thm:kset}.  Along with the notion of a $k$-division, 
the new ingredient needed is an oracle that for a given $m$, returns
a random solution of the modular subset sum problem over $\Z/m\Z$.
This oracle will be the multi-set birthday algorithm of Wagner \cite{Wagner02}, 
modified for the subset sum problem by Lyubashevsky \cite{Lyu05}
and proven correct in \cite{Sha08} (with complete proofs in \cite{thesis}).
In this section we present the multi-set birthday algorithm, modified 
to output a modular solution uniformly at random.  In Section \ref{sec:kset}
we demonstrate how this applies to the integer fixed weight subset sum problem
to finish the proof of Theorem \ref{thm:kset}.

Suppose we have lists $L_1, \dots, L_k$ of $N$ elements drawn uniformly and independently
from $\Z/m\Z$ and a target $t$.
The $k$-set birthday problem is to find $s_i \in L_i$ with $\sum s_i = t \mod{m}$.
We can assume without loss of generality that our target is $0$, since if it is
not we can replace $L_k$ with $L_k - t \mod{m}$ 
and the elements will still be uniformly generated from $\Z/m\Z$.  
Use the representation that places elements in the interval $[-\frac{m}{2}, \frac{m}{2})$.

We will now briefly describe the original $k$-set algorithm from \cite{Wagner02}. 
Assume that $k$ is a power of $2$, and define parameter
$p = m^{-1/(\log{k}+1)}$.   Let $I_0$ denote the interval $[-\frac{m}{2}, \frac{m}{2})$ 
and in general
let $I_{\lam}$ denote the interval $[-\frac{mp^{\lam}}{2}, \frac{mp^{\lam}}{2})$.  Denote
by $\bowtie_I$ the list merging operator, so that $L_1 \bowtie_I L_2$ is the 
set of elements $a+b \in I$ where $a \in L_1$, $b \in L_2$ and addition is in $\Z$.  
Let $\bowtie$ be the matching 
operator, so that $L_1 \bowtie L_2$ outputs pairs $(a,b)$ with 
$a \in L_1$, $b \in L_2$
such that $a + b = 0$ \ (over $\Z$).

These operators are instantiated as follows.  For $\bowtie_I$, start by sorting $L_1$ and $L_2$.
For each $a \in L_1$, search for $b$ from $L_2$ that fall in the interval $I - a$
 and place all such $a+b$ in the output list.  Note that 
if $L_1$ and $L_2$ have size $N$, then the complexity of this operator is $O(N \log{N})$ time 
and space.
For $\bowtie$, sort $L_1$ and apply a random permutation to $L_2$.  Then for each $b \in L_2$, 
search for $-b$ in $L_1$.  The complexity is again $O(N \log{N})$
time and space.

The $k$-set birthday algorithm proceeds as follows.  For level 
$\lambda$, $1 \leq \lam \leq \log{k}-1$, we denote lists by $L^{(\lam)}$ and 
apply the operator $\bowtie_{I_{\lam}}$ to pairs of lists.  At level $\log{k}$ we apply 
$\bowtie$ to the remaining pair of lists, and every
element of $L_1^{(\log{k})} \bowtie L_2^{(\log{k})}$ is a solution to the problem.
Here we deviate from Wagner slightly and have the algorithm 
output a random element from the result of $\bowtie$ to ensure that the output
is a random modulo $m$ solution.  Pseudocode for this algorithm appears as 
Algorithm \ref{mod_oracle}.

\begin{algorithm}
\caption{Modular $k$-set Oracle}
\label{mod_oracle}
\begin{algorithmic}[1]
\STATE {\bf Input:} Lists $L_1, \dots, L_k$ of size $N$, modulus $m$, target $t$
\STATE {\bf Output:} $s_1, \dots, s_k$ with $s_i \in L_i$ such that $s_1+\cdots+s_k-t = 0 \mod{m}$.
\STATE Set $p = m^{-1/(\log{k}+1)}$, ensure that $N > 1/p$
\STATE For all list elements use representation in $[-\frac{m}{2}, \frac{m}{2})$
\FOR{ level $\lam = 1$ to $\log{k}-1$ }
\STATE apply $\bowtie_{I_{\lam}}$ to pairs of lists
\ENDFOR
\STATE apply $\bowtie$ to the final pair of lists $(L_1^{(\log{k})}, L_2^{(\log{k})})$
\STATE output an element of $L_1^{(\log{k})} \bowtie L_2^{(\log{k})}$ at random
\end{algorithmic}
\end{algorithm}

We assume that with $N = 1/p$, the size of $L_1^{(\lam)} \bowtie_{I_{\lam}} L_2^{(\lam)}$ 
is again a list of size $1/p$ for $1 \leq \lam \leq \log{k}-1$.
In \cite{Sha08} it is proven that list elements at all levels are close to uniform.  
Furthermore, if we assume the initial lists have size $\alpha/p$ and modify the listmerge
operator so that for each $a \in L_1$, exactly one $b$ from $L_2$ is chosen so that 
$a+b \in I$, then $L_1^{(\lam)} \bowtie_{I_{\lam}} L_2^{(\lam)}$ again has $\alpha/p$ 
elements (with exponentially small failure probability).  
Here $\alpha$ is a parameter chosen that depends on the requested chance of 
failure; for our purposes it suffices to know it is bounded by a polynomial in $n$.

Now, our stated implementation of the listmerge operator keeps all sums $a+b \in I$ 
because we want all solutions to have a chance at being found.  Since having more elements
at each level only increases the probability of the $k$-set algorithm succeeding, we 
have a rigorously analyzed algorithm if we accept an additional complexity factor of 
$\alpha^{O(1)} = n^{O(1)}$.

With the two lists at level $\log{k}$ each having size $\alpha/p$ and containing (almost) uniform
elements in the interval $[-\frac{m^{2/(\log{k}+1)}}{2}, \frac{m^{2/(\log{k}+1)}}{2})$,  
we conclude by the work in \cite{NishSib88} 
that $L_1^{(\log{k})} \bowtie L_2^{(\log{k})}$ contains at least 
one element with positive probability,  and thus that 
Algorithm \ref{mod_oracle} outputs a solution with positive probability.
The complexity of the algorithm is the complexity of running $\bowtie_I$ a 
total of $2k$ times, for a total of $\Ot(m^{1/(\log{k}+1)})$ time and space.

\subsection{Randomizing the Modular Oracle}\label{subsect:random}

Note that not every solution to the modular subset sum problem could be output
by Algorithm \ref{mod_oracle}.  Inspired by a suggestion from \cite{Wagner02}, 
our focus for the rest of this section
will be on using Algorithm \ref{mod_oracle} to generate a random solution to the $k$-set 
birthday problem, one which has a nearly uniform distribution.

Define the $2$-sums of the problem to be 
$L_1+L_2, L_3+L_4, \dots, L_{k-1}+L_k$, the $4$-sums to be 
$L_{4i+1} + L_{4i+2}+L_{4i+3}+L_{4i+4}$ for $0 \leq i \leq \frac{k-4}{4}$, 
and so on up to the two $k/2$-sums
$L_1+\cdots+L_{k/2}$ and $L_{k/2+1}+\cdots+L_k$.  This term will also be 
used for the corresponding sums of a particular solution $(s_1, \dots, s_k)$.
We refer to both integer sums and modular sums depending on whether 
the addition is over $\Z$ or over $\Z/m\Z$.

Let $R$ be a set of $\frac{3k}{4} - 1$ elements of $\Z/m\Z$ generated uniformly 
at random.  For each of the $4$-sums, replace the lists $L_1$, $L_2$, $L_3$, $L_4$
with $L_1+r_1$, $L_2+r_2$, $L_3-r_1$, $L_4-r_2$ where $r_1$ and $r_2$ 
are two elements of $R$.
For each of the $8$-sums, replace $L_{8i+4}$ with $L_{8i+4}+r$ and $L_{8i+8}$
with $L_{8i+8}- r$.  In general, for each of the $2^j$-sums ($3 \leq j \leq \log{k}$), 
replace $L_{2^j i + 2^{j-1}}$ with $L_{2^j i + 2^{j-1}} + r$ and 
$L_{2^j i + 2^j}$ with $L_{2^j i + 2^j} - r$.
All these operations are in $\Z/m\Z$.

In the example of the $8$-set 
algorithm $R = \{r_1, r_2, r_3, r_4, r_5\}$ and lists $L_1, \dots, L_8$ are replaced by
$$
L_1+r_1, L_2 + r_2,  L_3-r_1,  L_4 - r_2 + r_5, 
 L_5+r_3, L_6 + r_4, L_7-r_3, L_8 -  r_4 - r_5 \enspace .
$$

We seek to prove that applying Algorithm \ref{mod_oracle} to lists modified in this way
results in a solution drawn almost uniformly at random from the space of all solutions to the 
$k$-set birthday problem on fixed lists $L_1, \dots, L_k$.  
To classify which solutions are possibly output we make the following definition.

\begin{definition}
Let a solution $s_1+\cdots+s_k$ modified in the above manner by a randomizing set $R$
be denoted $s_1'+\cdots+s_k'$.
Call a solution to the modulo $m$ subset sum problem \emph{viable} with respect to 
a randomizing set $R$ if for $1 \leq i \leq \log{k}-1$, all integer $2^i$-sums $s'$ satisfy 
$s' \in I_i$.
\end{definition}

We will also refer to an individual integral or modular
 $2^i$-sum $s'$ as viable if $s' \in I_i$.

Algorithm \ref{mod_oracle} performs additions in $\Z$ despite the fact that a 
modular solution is sought.  Our goal is to prove that the number of 
randomizing sets making a solution $s$ viable is roughly equal.  We first 
prove this for modular $2^i$ sums with $i \geq 2$ in Lemma \ref{general_sums} 
starting with the $\frac{k}{2}$-sums and working down. 
The integer $2$-sums are analyzed in Lemma \ref{4_sums}, from which
the main theorem quickly follows.  The key observation is that a modular solution 
with viable modular $2^i$-sums for all $i > 2$ and viable integral $2$-sums
must also have viable integral $2^i$-sums for all $i > 2$.

\begin{lemma}\label{general_sums}
Let $\frac{k}{2^i} > 2$ and consider $s+t$, the sum of two $\frac{k}{2^i}$-sums.
Assuming that $s+t \mod{m} \in I_{\log{k}-i+1}$, the number of $r$
 such that $s+r \mod{m}$
and $t-r \mod{m}$ simultaneously fall in $I_{\log{k}-i}$ is at least
$mp^{\log{k}-i} (1-p) - 1$ and at most $mp^{\log{k}-i}$. 
\end{lemma}
\begin{proof}
Call an $r$ value good if $s+r \mod{m} \in I_{\log{k}-i}$ and 
$t-r \mod{m} \in I_{\log{k}-i}$.

The maximum number of good $r$ values occurs when $s+t = 0 \mod{m}$.  
The size of $I_{\log{k}-i}$ is $\floor{mp^{\log{k}-i}}$, and so this is the 
number of $r$ such that $s+r \mod{m} \in I_{\log{k}-i}$.  Since 
$s = -t \mod{m}$, the same set of $r$ place $r -t \mod{m} \in I_{\log{k}-i}$, 
and the same set of $r$ place $t-r \mod{m} \in I_{\log{k}-i}$ since the interval
is symmetric.

The minimum occurs when
$s+t = \pm mp^{\log{k}-i+1}$.  The number of $r \in \Z/m\Z$ that place 
$s+r \in I_{\log{k}-i}$ is  $\lfloor mp^{\log{k}-i} \rfloor$.  The same set 
of $r$ values place $r-t + mp^{\log{k}-i+1} \in I_{\log{k}-i}$, but a total of 
$mp^{\log{k}-i+1}$ of the $r$ values are lost when we instead ask for 
$t-r \mod{m} \in I_{\log{k}-i}$.  So the number of valid $r$ values is at least 
$\lfloor mp^{\log{k}-i} -  mp^{\log{k}-i+1} \rfloor \geq
   mp^{\log{k}-i} (1-p) - 1$.
\end{proof}

Suppose that randomizers have been found that place the modular 
$4$-sums of a solution in $I_{2}$.  We now seek to place
the integer $2$-sums in $I_1$.  Since we will be mixing integer addition 
and modular addition, we use $\oplus$ to signify the latter.  Recall that we 
are using $[-\frac{m}{2}, \frac{m}{2})$ as the set of representatives for elements 
of $\Z/m\Z$.

\begin{lemma}\label{4_sums}
Suppose that $s_1+s_2+s_3+s_4 \mod{m}$ is in $I_{2}$.
Then the number of pairs $(r_1, r_2)$ such that 
$$(s_1 \oplus  r_1) + (s_2 \oplus r_2)  \in I_1
\mbox{ and } (s_3 \ominus r_1) + (s_4 \ominus r_2) \in I_1
$$
is at most $m^2p$ and at least $(m - 2mp)(mp - mp^2-1)$.
\end{lemma}
\begin{proof}
First, consider a fixed $r_1$, and let $s_1' = s_1 \oplus r_1$  and $s_3' = s_3 \ominus r_1$.
Then we need $s_2 \oplus r_2 \in I_1 - s_1'$, where the interval subtraction is over $\Z$.
The size of $I_1 - s_1'$ might be as small as $\frac{mp}{2}$ if $s_1' = \pm \frac{m}{2}$.  
Since we can choose
$r_2$ such that $s_2 \oplus r_2$ is any element in $[-\frac{m}{2}, \frac{m}{2})$, the number 
of such $r_2$ is the size of $I_1 - s_1'$.
Simultaneously $r_2$ must satisfy $s_4 \ominus r_2 \in I_1 - s_3'$.
There are two extremes, depending on whether $s_1+s_2+s_3+s_4 = 0 \mod{m}$ or 
$s_1 + s_2 + s_3 + s_4 = \pm mp^2 \mod{m}$.

In the first case, $s_1' \oplus s_2 = -(s_3' \oplus s_4)$ and $I_1$ symmetric implies that 
there are at most $mp$ values of $r_2$ such that $s_1' \oplus s_2 \oplus r_2, 
s_3' \oplus s_4 \ominus r_2$ are in $I_1$. 
Since switching to $s_3' + s_2 \oplus r_2$ and $s_3' + s_4 \ominus r_2$ can only 
reduce the number of valid $r_2$, $mp$ is an upper bound.

However, if $s_1' \oplus s_2 = mp^2 \ominus (s_3' \oplus s_4)$, 
then by the same argument from Lemma \ref{general_sums} the number 
of valid $r_2$ for the modular sums is $\floor{mp - mp^2}$.  The number 
of valid $r_2$ for the integer sums could be smaller depending on the sizes 
of $I_1 - s_1'$ and $I_1 - s_3'$.

Now consider the size of $I_1 - s_1'$ and $I_1 - s_3'$ depending on $r_1$.  When $r_1$ shifts
by one, the intervals shift by one as well.  The intervals will have less than full size
when $s_1 \oplus r_1$ or $s_3 \ominus r_1$ is less than $-\frac{m}{2} + \frac{mp}{2}$
or greater than $\frac{m}{2} - \frac{mp}{2}$.  Hence the number of $r_1$ that make for one 
of the intervals to have less than full size is at most $2mp$.

Thus the number of valid pairs $(r_1, r_2)$ is at most $m^2p$ (assuming intervals 
full size for all $r_1$ and in case one above) and is at least 
$(m - 2mp)(mp - mp^2 - 1)$ (assuming interval size taken from case two).
\end{proof}

\begin{theorem}\label{oracle_prob}
Assume that $\frac{p}{2} < \frac{1}{k}$.
Let $A$ be the event that a solution $s = s_1+\cdots+s_k$ is output by the modular oracle, given
that some solution is output.  Then the distribution of $A$ is uniform within a factor of 
$(1-2p)^{3k/4}$.
\end{theorem}
\begin{proof}
We have $\Pr[s \mbox{ solution}] = 
\Pr[s \mbox{ viable}]\Pr[s \mbox{ solution} \ | \ s \mbox{ viable}]$, 
where we leave unwritten the assumption that some solution is output.
We first bound $\Pr[s \mbox{ viable}]$.

We have $s_1 + \cdots + s_k = 0 \mod{m}$.  Using the same argument as in the first 
case of Lemma \ref{general_sums}, there are $mp^{\log{k}-1}$ values of $r$ such 
that $s_1 + \cdots + s_{k/2} + r \mod{m}$ and $s_{k/2+1}+\cdots+s_k-r \mod{m}$
both fall in $I_{\log{k}-1}$.

Using this as the base case and Lemma \ref{general_sums} as the inductive step, 
we have upper and lower bounds on the number of randomizers at each level.
Given randomizers that place modular sums in the proper interval, and in particular
that place modular $4$-sums in $I_2$, Lemma \ref{4_sums} gives us the number 
of randomizers that place integer $2$-sums in $I_1$.  Thus our modified solution
$s_1'+\cdots+s_k'$ is a modular solution with integer $2$-sums, which since 
$k \cdot \frac{mp}{2} < m$ implies that all integer $2^i$-sums lie in $I_i$, and hence 
that the solution is viable with respect to those randomizing sets.

There are a total of $m^{3k/4-1}$ randomizing sets.  Combining the bounds 
from Lemmas \ref{general_sums} and \ref{4_sums} give the following bounds
on the number for which $s$ is viable.

Setting $N = \frac{k}{4} + 2 \frac{k}{8} + 3 \frac{k}{16} + \cdots + (\log{k}-1) \frac{k}{k}$
an upper bound is given by
$$
(m^2p)^{k/4} \cdot (mp^2)^{k/8} \cdot (mp^3)^{k/16} \cdots (mp^{\log{k}-1})  
= m^{3k/4-1} \cdot p^{N} \enspace.
$$
Noting that $mp^{\log{k}-i} -mp^{\log{k}-i+1}-1 \geq  mp^{\log{k}-i}(1-2p)$
a lower bound is given by
\begin{align*}
& (m^2p(1-2p)^2)^{k/4} \cdot (mp^2(1-2p))^{k/8} \cdot (mp^3(1-2p))^{k/16}
 \cdots (mp^{\log{k}-1})(1-2p) \\
&=  m^{3k/4-1} \cdot (1-2p)^{3k/4-1} \cdot p^{N} \enspace .
\end{align*} 

Thus $\Pr[s \mbox{ viable}]$ is uniform on the upper bound and 
uniform within a factor of $(1-2p)^{3k/4}$ on the lower bound.

We now consider the second term.
Algorithm \ref{mod_oracle} is written so that for a given set of randomizers, 
a solution is output uniformly at random from the set of viable solutions.  Since the 
number of viable solutions is bounded by $\Pr[s \mbox{ viable}]$ 
times the number of solutions, the fact that $\Pr[s \mbox{viable}]$ is close to uniform
makes $\Pr[s \mbox{ solution} \ | \ s \mbox{ viable}]$ close to uniform, but with the 
factors on the upper and lower bounds switched.

Thus upper and lower bounds for the probability of the event $A$ are separated 
from uniform by a factor 
of $(1-2p)^{3k/4}$.
\end{proof}

\section{The $k$-Set Algorithm}\label{sec:kset}

In this section we utilize the $k$-set modular oracle in designing an 
algorithm for the fixed weight subset sum problem.  Lyubashevsky \cite{Lyu05} was 
the first to leverage an algorithm for the modular subset sum 
problem out of an algorithm
for the $k$-set birthday problem.  Our modifications include dealing with the fixed weight
nature of the problem by employing a $k$-division, and dealing with the integral nature
of the problem by looping on the modular oracle until an integer solution is found.  The 
pseudocode appears as Algorithm \ref{alg:fixed_weight}.

\begin{algorithm}
\caption{Multi-set Algorithm for Fixed Weight Subset Sum}
\label{alg:fixed_weight}
\begin{algorithmic}[1]
\STATE {\bf Input:} positive integers $a_1, \dots, a_n$, target $t$, weight $\ell$, 
parameters $k$, $m$
\STATE {\bf Output:} $\vec{x} \in \{0,1\}^n$ of weight $\ell$ with 
$\sum_{i=1}^n a_i x_i = t$
\WHILE{no integer solution}
\STATE choose random $k$-division $(I_1, \dots, I_k)$
\STATE choose set $R$ of $\frac{3k}{4}-1$ random elements of $\Z/m\Z$.
\STATE form lists $L_1, \dots, L_k$ of size $m^{1/(\log{k}+1)}$ whose elements
are random subsets of weight $\ell/k$ from appropriate $I_j$, reduced modulo $m$
\STATE apply randomizers from $R$ to lists as described in Section \ref{subsect:random}
\STATE apply Algorithm \ref{mod_oracle} to $L_1, \dots, L_k$
\IF{success}
\STATE check if integer solution
\ENDIF
\ENDWHILE
\end{algorithmic}
\end{algorithm}

If $\ell$ is small compared to $k$, one could instead solve 
the $(n-\ell)$-weight subset sum problem with target $(\sum_{i=1}^n a_i) - t$.

Algorithm \ref{mod_oracle} takes as input uniformly distributed elements of 
$\Z/m\Z$.  By the work in \cite{ImpNao96}, if $a_1 \mod{m}, \dots, a_n \mod{m}$
are uniformly distributed over $\Z/m\Z$, then random $n/k$-length, $\ell/k$-weight
 subsets of these elements 
will be exponentially close to uniform as long as $m < {n/k \choose \ell/k}$.
If in addition we seed the lists with $poly(n) \cdot m^{1/(\log{k}+1)}$ elements, 
then combined with 
the work of Section \ref{subsect:random} we get a rigorous analysis 
of Algorithm \ref{alg:fixed_weight}.

We now prove Theorem \ref{thm:kset} (restated here for convenience)
by analyzing Algorithm \ref{alg:fixed_weight}.
To solve the integer fixed weight subset sum problem, we make an appropriate 
choice of $m$, which determines the resulting point on the time-space tradeoff curve.
The necessary assumption that $\frac{p}{2} < \frac{1}{k}$ in Theorem \ref{oracle_prob}
is satisfied by choosing $m$ and $k$ so that $\log{m} \geq 2(\log{k})^2$.

\begin{theorem}
Choose parameters $m$ and $k$ so that $k$ is a power of $2$, 
$m < {n/k \choose \ell/k}$, and $\log{m} \geq 2(\log{k})^2$.
Assume that when reduced modulo $m$, 
the $a_i$ are uniformly random elements of $\Z/m\Z$.
Then the expected running time of Algorithm \ref{alg:fixed_weight} is
$\Ot(m^{1/(\log{k}+1)} \cdot {n \choose \ell}/m)$ and the algorithm uses 
$\Ot(m^{1/(\log{k}+1)})$  space.  This gives a point on the time/space tradeoff curve
$T \cdot S^{\log{k}} = {n \choose \ell}$.
\end{theorem}
\begin{proof}
The probability that Algorithm \ref{alg:fixed_weight} finds a solution on a particular
interation of the {\bf while} loop is the product of three probabilities:
the probability that the $k$-division is good with respect to some unknown solution,  
the probability that Algorithm \ref{mod_oracle} succeeds, 
and the probability 
that the modular solution found by Algorithm \ref{mod_oracle} 
is also the integer solution.

By Proposition \ref{rand_splitting} the first term is greater than 
$\ell^{\frac{1-k}{2}}$.  The second probability is greater than some fixed 
$\epsilon$ by the previous work outlined in Section \ref{sec:mod_oracle}.
For the third term, we first call upon a theorem of Implagliazzo and Naor
\cite{ImpNao96} (proven using the leftover hash lemma) which tells us that 
with the $a_i$ drawn uniformly at random from $\Z/m\Z$ and $m < {n \choose \ell}$, 
the distribution of random $\ell$-weight subsets is exponentially close 
to uniform.  Thus we expect the number of 
modular solutions to be a constant times ${n \choose \ell}/m$. 
By Theorem \ref{oracle_prob} we conclude that 
the third probability factor is greater than 
$(1-2p)^{3k/4} \cdot m/{n \choose \ell}$.  
Note that $(1-2p)^{3k/4} \geq 1 - \frac{3k}{2} p \geq \frac{1}{2}$
since $\log{m} \geq 2(\log{k})^2$ implies 
$p = m^{-1/(\log{k}+1)} \leq \frac{1}{3k}$.

Thus the expected number of iterations of the {\bf while} loop is 
$$
O\left(\epsilon \ell^{\frac{k-1}{2}} \cdot \left. 2 {n \choose \ell}
\right/ m  \right)  \enspace .
$$

The cost of each iteration is dominated by Algorithm \ref{mod_oracle}, which 
takes $\Ot(m^{1/(\log{k}+1)})$ time and space.

Thus Algorithm \ref{alg:fixed_weight} takes expected time 
$\Ot(m^{1/(\log{k}+1)} \cdot {n \choose \ell}/m)$ and space
$\Ot(m^{1/(\log{k}+1)})$, which is a point on the time and space tradeoff curve 
$T \cdot S^{\log{k}} = {n \choose \ell}$.
\end{proof}

As an example of parameter choices in action, suppose we wish to solve an 
integer fixed weight subset sum problem with an $8$-set birthday algorithm.
Our conjectural maximal choice of $m$ is ${n/8 \choose \ell/8}^4$, which is 
approximately ${n \choose \ell}^{1/2}$.  Thus we expect the problem to be solved
in time $\Ot({n \choose \ell}^{1/8}{n \choose \ell}^{1/2})$ and space 
$\Ot({n \choose \ell}^{1/8})$.

Note that Algorithm \ref{alg:fixed_weight} is highly parallelizable, since running it 
simultaneously on $N$ processors increases the probability of success by 
a factor of $N$. 

\section{Application to Knapsack Cryptosystems}\label{sec:knapsack}

Knapsack cryptosystem is the term used for a class of public key cryptosystems 
whose underlying hard problem is the integer subset sum problem. 
Though few have remained unbroken, 
the search for knapsack cryptosystems remains popular due to 
their fast encryption and easy implementation.

A knapsack cryptosystem is defined abstractly as follows.  We have a public key
$(a_1, \dots, a_n)$ defining a hard subset sum problem, and a private key which transforms 
the hard problem into an easy subset sum problem.  To send a message $\vec{x} \in \{0,1\}^n$, 
a user computes $t = \sum_{i=1}^n a_i x_i$ and sends it.  The receiver, who has the private key,
transforms the problem and then solves the easy subset sum problem to recover $\vec{x}$.

There are two main attacks on knapsack cryptosystems.  First, there are key attacks which 
attempt to recover the easy subset sum problem from the public key.  Second, there are message
attacks which attempt to recover the message by solving the hard subset sum problem
$a_1 x_1 + \cdots + a_n x_n = t$.  
Key attacks are not our concern in this paper, we simply note that many systems
have succumbed to such attacks, the seminal cryptosystem of Merkle-Hellman \cite{MH78}
among them.
We focus instead 
on message attacks, which are equivalent to solving the subset sum problem or its variants.

The most successful message attack in theory and in practice is the low-density attack that 
reduces the subset sum problem to the shortest vector problem or the closest vector problem, 
discussed in Section \ref{sec:results}.
Since unique decryption requires $2^n \leq \sum_{i=1}^n a_i$, and hence that the density be 
no more than a little above $1$, these results pose a conundrum for the knapsack designer.
As a result, modern designs  have relied on fixing the hamming weight of allowed messages, 
so that the underlying hard problem becomes the fixed weight subset sum problem.  
This began with Chor-Rivest \cite{ChRiv88} and continues into the present 
with the notable OTU scheme \cite{OTU00} and its non-quantum variant \cite{KateGold07}.
In this way
$n$ can be made great enough so that the density is above one, while the information 
density stays below one to preserve unique decryption.  As an added bonus, the fixed weight 
subset sum problem has received much less attention in the literature, and so 
message attacks
remain in a primitive state.  Until recently the only known algorithm was 
the square root time-space
tradeoff algorithm in \cite[Section 7.3]{ChRiv88}.  

Here we have only scratched the surface of the vast literature on knapsack cryptosystems.
For further information consult the survey \cite{Odl90}.

The new result in this paper is Theorem \ref{thm:SS} from which we immediately get a 
message attack that takes square root time and fourth root space.  Theorem \ref{thm:kset}
is less interesting from this perspective because the large constant and polynomial terms, 
along with the sharp upper bound on the size of $m$, mean that seldom would the 
$k$-division algorithm reach even square root time in practice.

\section{Data and Conclusions}\label{sec:data}

In this section we explore experimentally two questions related to 
Algorithm \ref{alg:fixed_weight}.  The first is to measure the number 
of times Algorithm \ref{mod_oracle} succeeds before an integer solution is 
found, and to compare that to the expected number ${n \choose \ell}/m$.
The second is to measure the success probability of Algorithm 
\ref{mod_oracle} when the modular information density is pushed lower 
than Theorem \ref{thm:kset} requires.  In particular, $m$ cannot be larger than 
${n/k \choose \ell/k}^{\log{k}+1}$ since otherwise there will not be enough 
weight $\ell/k$ subsets to fill the lists $L_j$, so we choose $m$ between 
${n/k \choose \ell/k}$ and ${n/k \choose \ell/k}^{\log{k}+1}$.

We implemented $2$-set, $4$-set, and $8$-set algorithms for the modular subset sum 
problem and applied them to the integer subset sum problem.  We chose not to explore
the additional impact of searching for a $k$-division, since the probability 
calculation is straightforward.
We ran these algorithms on a desktop workstation 
on problems with $n$ equal to $24$ and an integer density
of $0.9$.

In the tables that follow $d_m$ denotes the modular density.  Each entry represents
the mean over ten trials, except those marked with a $*$ which represent  
the result after one trial.  Let $N_o$ be the number of modular oracle 
successes before an integer solution is found.

\bigskip

\begin{tabular}{|c|c|c|c|c|c|c|}
\hline
& \multicolumn{2}{|c|}{$d_m = 1.5$} & \multicolumn{2}{|c|}{$d_m = 2$} & 
            \multicolumn{2}{|c|}{$d_m=4$} \\ \hline
& $N_o$ & $\E[N_o]$ & $N_o$ & $\E[N_o]$ & $N_o$ & $\E[N_o]$ \\ \hline
$2$-set & $209$ & $256$ & $1955$  & $4096$ & $353000$ & $262000$ \\ \hline
$4$-set & $168$ & $256$ & $5436$  & $4096$ & $260000$ & $262000$ \\ \hline
$8$-set & $265^*$ & $256$ & $1831$  & $4096$ & $330000$ & $262000$ \\ \hline
\end{tabular}
 
\bigskip

The next table explores the effect that parameters $m$ and $k$ have on 
Algorithm \ref{alg:fixed_weight}.

\bigskip

\begin{tabular}{|c|c|c|c|c|c|c|}
\hline
& \multicolumn{2}{|c|}{$d_m = 1.5$} & \multicolumn{2}{|c|}{$d_m =2$} & 
         \multicolumn{2}{|c|}{$d_m = 4$} \\ \hline
& success \% & time (s) & success \% & time (s) & success \% & time (s) \\ \hline
$2$-set & $58.9$ \% & $15$ & $61.4$ \% & $28$ & $58.1$ \% & $466$ \\ \hline
$4$-set & $19.8$ \% & $121$ & $40.5$ \% & $336$ & $46.7$ \% & $1594$ \\ \hline
$8$-set & $0.7^*$ \% & $11058^*$ & $11.9$ \% & $945$ & $57.2$ \% & $6069$ \\ \hline
\end{tabular} 

\bigskip

Taken together, this data supports our heuristic analysis of Algorithm 
\ref{alg:fixed_weight}.  We see that the modular oracle succeeds with
some constant probability, and that the number of successful oracle calls needed
is roughly the expected number (though the variance is quite large).

We also see that despite a lower success percentage, choosing 
$d_m$ as small as possible results in a faster running time.
There is a boundary beyond which the algorithm succeeds too rarely to 
be of any use, as exemplified by the $8$-set algorithm with $d_m = 1.5$.
A reasonable conjecture places this boundary 
at $d_m = \frac{k}{\log{k}+1}$, since below this point, there are not 
enough subsets to fill the lists $L_1, \dots, L_k$ with $m^{1/(\log{k}+1)}$ 
elements.  

As $k$ increases the overhead associated with the more complicated algorithms outstrips 
their asymptotic improvement, at least for $n=24$.  It is unclear how large $n$
will have to be before the $8$-set algorithm is faster than the $2$-set algorithm
for $d_m=4$.  

\section{Splitting Systems in the Indivisibility Case}\label{sec:gen_splitting}

In Section \ref{sec:splitting} we presented $(n, \ell, k)$-splitting systems
and proved their existence under the assumption that $n$ and $\ell$ were divisible by $k$.
In this section we relax this restriction, showing that splitting systems 
exist when $n$, $\ell$ are any positive integers greater than $k$.  Let 
positive integers 
$r_1$, $r_2$ be defined by 
 $n = k \cdot \floor{n/k} + r_1$ and $\ell = k \cdot \floor{\ell/k} + r_2$.

\begin{definition}
A \emph{$(n, \ell, k)$-splitting system} is a set $X$ of $n$ indices along with a set 
$\mathcal{D}$ of divisions, where each division is itself a set $\{I_1, \dots, I_k\}$
of subsets of indices.  Here the $I_j$ partition $X$ and their sizes satisfy
$|I_1| = \cdots = |I_{k-1}| = \floor{n/k}$, $|I_k| = \floor{n/k}+r_1$.  A splitting system
has the property that for every $Y \subseteq X$ such that $|Y| = \ell$, there exists a division
$\{I_1, \dots, I_k\} \in \mathcal{D}$ such that $|Y \cap I_j| = \floor{\ell/k}$ for 
$1 \leq j \leq k-1$ and $|Y \cap I_k| = \floor{\ell/k}+r_2$.
\end{definition}

Again, with $n, \ell, Y$ understood as parameters of a fixed weight subset sum problem 
we are interested in solving, we refer to an $(n, \ell, k)$-splitting system as a $k$-set 
splitting system.

Most likely a better strategy in practice would be to
spread the extra weight among the $I_j$ rather than assigning it all to $I_k$.  
This definition was chosen to quickly demonstrate that nondivision poses no 
barrier in theory.
The key result is to prove the existence of this more general structure.  In order to do 
this, we will first find $I_1, \dots, I_{k-1}$, and leave the remainder of $X$ to $I_k$.
Our candidates will be 
$$B_i^{(n)} = \{i+j \mod{n} \ | \ 0 \leq j \leq \floor{n/k} \} \enspace .$$
Given a fixed $Y \subset X$ of size $\ell$, we 
define a function $\nu$ be $\nu(i) =  |B_i \cap Y| - \floor{\ell/k} $.

\begin{proposition}
There exists a $k$-set splitting system with fewer than $n^{k-1}$ divisions.
\end{proposition}
\begin{proof}
Our initial goal is to prove that there must exist an $i$ with $\nu(i) = 0$.
Consider $B_0$, $B_{\floor{n/k}}$, $B_{2 \floor{n/k}}, \dots, B_{(k-2)\floor{n/k}}$.
Define $B$ to be the remainder of the indices of $X$.  If $\nu(i) = 0$ for one of
$i = 0, \floor{n/k}, \dots, (k-2)\floor{n/k}$ then we are done.  If not, 
we wish to find $i, i'$ such that $\nu(i), \nu(i')$ have opposite signs.

If $\nu(i) > 0$ for each of 
$i = 0, \floor{n/k}, \dots, (k-2)\floor{n/k}$, then the combined weight
of the corresponding $B_i$ is at least $(k-1) \floor{\ell/k} + k-1$ and so $B$ must have weight
less than $\floor{\ell/k} + r_2 - (k-1) \leq \floor{\ell/k}$.  
Thus in particular $B_{(k-1)\floor{n/k}}$, 
the first $\floor{n/k}$ indices of $B$, must have 
weight less than $\floor{\ell/k}$.   

If $\nu(i) < 0$ for $i = 0, \floor{n/k}, \dots, (k-2)\floor{n/k}$, then the combined weight 
of the corresponding $B_i$ is at most $(k-1) \floor{\ell/k} - (k-1)$ and so $B$ must have 
weight greater than $\floor{\ell/k} + r_2+k-1$.  Then $B_{(k-1)\floor{n/k}}$, the first 
$\floor{n/k}$ indices of $B$, must have weight greater than $\floor{\ell/k}$.  For if not, 
the weight of $B$ is at most $\floor{\ell/k}+r_1 \leq \floor{\ell/k}+r_2+k-1$, a contradiction.

In either case there is an $i$ with $\nu(i) > 0$ and an $i'$ with $\nu(i') < 0$.  Since 
\linebreak $| \nu(i) - \nu(i+1)| \leq 1$, there must be an $i$ with $\nu(i) = 0$.  
Label the corresponding set $I_1$.

We now remove the indices in $I_1$ from consideration, relabel the indices
$0, \dots, n - \floor{n/k}$, and seek an $i$ such that $B_i^{(n - \floor{n/k})}$ has weight
$\floor{\ell/k}$.  Using the same reasoning as above, one must exist.

In this fashion $I_1, \dots, I_{k-1}$ can be found.  The remaining indices make up $I_k$.
The number of divisions needed to satisfy this process is 
$$
n(n - \floor{n/k})(n - 2 \floor{n/k}) \cdots (n - (k-2)\floor{n/k}) < n^{k-1} \enspace .
$$
\end{proof}

Next we discuss the effect on running times.  For the Shroeppel-Shamir algorithm, 
the main terms of the complexity bounds become
 ${\floor{n/4}+3 \choose \floor{\ell/4}+3}^2$
time and ${\floor{n/4}+3 \choose \floor{\ell/4}+3}$ space.  Since
$$
{\floor{n/4}+3 \choose \floor{\ell/4}+3} = \left(\frac{n}{\ell}\right)^3 
{\floor{n/4} \choose \floor{\ell/4}}
$$
the complexity is worse by at most a polynomial factor.  A similar result holds for 
the modular oracle.  The polynomial factor becomes $(n/\ell)^k$, which is polynomial 
for constant $k$.

\bibliography{knapsack}
\bibliographystyle{amsplain}

\end{document}